\documentclass[12pt]{article}
\usepackage{amsmath}
\allowdisplaybreaks[4]
\usepackage{tikz}
\usetikzlibrary{decorations.pathreplacing, matrix}
\usetikzlibrary{decorations.pathreplacing}
\usepackage[all]{xy}
\usepackage{amssymb}
\usepackage{amsthm}
\usepackage{hyperref}
\usepackage{bm}
\hypersetup{colorlinks=true,linkcolor=blue,citecolor=red}
\usepackage{amsmath}
\usepackage{amscd}
\usepackage{verbatim}
\usepackage{eurosym}
\usepackage{float}
\usepackage{color}
\usepackage{dcolumn}
\usepackage[mathscr]{eucal}
\usepackage[all]{xy}
\usepackage{hyperref}
\usepackage{mathrsfs}
\usepackage{amsmath}
\usepackage{amssymb}
\usepackage{amsfonts,ifpdf}
\usepackage{graphicx}
\usepackage{times}
\usepackage{float}
\usepackage{epstopdf}
\usepackage{cite}
\usepackage{youngtab}
\usepackage{ytableau}
\ytableausetup
{mathmode, boxsize=0.9em}
\usepackage{indentfirst}

\newtheorem{theorem}{Theorem}[section]
\newtheorem{lemma}[theorem]{Lemma}

\newtheorem{coro}[theorem]{Corollary}

\theoremstyle{definition}
\newtheorem{defi}[theorem]{Definition}

\newtheorem{pf}{proof}[section]
\theoremstyle{remark}

\makeatletter \@addtoreset{equation}{section} \makeatother
\makeindex 

\def\rk{\operatorname{rk}}

\renewcommand\ell{l}

\def\det{\mbox{\rm det}}
\def\And{\mbox{\rm ~and~}}

\begin{document}
\begin{center}
{\Large\bf
A decomposition of Grassmannian associated with a hyperplane arrangement
}\\ [7pt]
\end{center}
\vskip 3mm
\begin{center}
Houshan Fu$^1$,  Weikang Liang$^{2,*}$ and Suijie Wang$^3$\\[8pt]

$^1$School of Mathematics and Information Science\\
Guangzhou University\\
Guangzhou 510006, Guangdong, P. R. China\\[15pt]
		
$^{2,3}$School of Mathematics, $^3$Greater Bay Area Institute for Innovation\\
Hunan University\\
Changsha 410082, Hunan, P. R. China\\[15pt]
				
$^*$Correspondence to be sent to: kangkang@hnu.edu.cn  \\
Emails: $^1$fuhoushan@gzhu.edu.cn, $^3$wangsuijie@hnu.edu.cn\\[15pt]
\end{center}
\vskip 3mm
\begin{abstract}
The Grassmannian, which is the manifold of all $k$-dimensional subspaces in the Euclidean space $\mathbb{R}^n$, was decomposed through three equivalent methods connecting combinatorial geometries, Schubert cells and convex polyhedra by Gelfand, Goresky, MacPherson and Serganova. Recently, Liang, Wang and Zhao discovered a novel decomposition of the Grassmannian via an essential hyperplane arrangement, which generalizes the first two methods. However, their work was confined to essential hyperplane arrangements. Motivated by their research, we extend their results to a general hyperplane arrangement $\mathcal{A}$, and demonstrate that the $\mathcal{A}$-matroid, the $\mathcal{A}$-adjoint and the refined $\mathcal{A}$-Schubert decompositions of the Grassmannian are consistent. As a byproduct, we provide a classification for $k$-restrictions of $\mathcal{A}$ related to all $k$-subspaces through two equivalent methods: the $\mathcal{A}$-matroid decomposition and the $\mathcal{A}$-adjoint decomposition. 
\\ [3pt]
\end{abstract}
\section{Introduction}
The {\em Grassmannian} $\mathrm{Gr}(k, n)$  is the manifold of all $k$-dimensional linear subspaces in the Euclidean space $\mathbb{R}^n$. Gelfand, Goresky, MacPherson and Serganova \cite{Gelfand-Goresky-MacPherson-Serganova1987} introduced three equivalent approaches to decompose the Grassmannian, which connect the geometry of Schubert cells in the Grassmannian manifold, the theory of combinatorial geometries (matroid theory), and the theory of convex polyhedra. Inspired by their work, Liang, Wang and Zhao \cite{Liang-wang-zhao} recently discovered a decomposition of the Grassmannian through an essential hyperplane arrangement $\mathcal{A}$, which can reduce to the first two approaches when considering $\mathcal{A}$ as the Boolean arrangement. However, their analysis was confined to essential hyperplane arrangements. Building on their work, our main goal in this paper is to extend their main results to general hyperplane arrangements.

To be more precise, a {\em hyperplane arrangement} $\mathcal{A}$ is a finite collection of hyperplanes in the Euclidean space $\mathbb{R}^n$. Throughout this paper, we assume that $\mathcal{A}$ is {\em linear}, i.e., all the hyperplanes in $\mathcal{A}$ pass through the origin $o$. Let $T$ denote the intersection $T:=\bigcap_{H\in\mathcal{A}}H$. In particular, $\mathcal{A}$ is {\em essential} if $T=\{o\}$. The {\em intersection lattice} $L(\mathcal{A})$ is a poset consisting of all nonempty intersections of some hyperplanes in $\mathcal{A}$, ordered by reverse inclusion, and including the whole space $\mathbb{R}^n=\bigcap_{H\in\emptyset}H$ as the (unique) minimal element. It is well known that $L(\mathcal{A})$ is a geometric lattice and its {\em rank function} is given by 
\[
\rk(X):=n-\dim(X),\quad\forall\, X\in L(\mathcal{A}).
\]
Especially, we call $r(\mathcal{A}):=\rk(T)$ the {\em rank} of $\mathcal{A}$.  Moreover, we call an $(n-k)$-dimensional element in $L(\mathcal{A})$ a {\em $k$-flat} of $\mathcal{A}$. Naturally, the poset $L(\mathcal{A})$ can be decomposed into $r(\mathcal{A})+1$ parts, with every part $L_{k}(\mathcal{A})$ consisting of all $k$-flats
\[L_{k}(\mathcal{A}):=\big\{X \in L(\mathcal{A})  \mid \rk(X)=k\big\},\quad k=0,1,\ldots,r(\mathcal{A}).\]

Extending the restriction operator of the hyperplane arrangement $\mathcal{A}$ to an arbitrary subspace $U$ of $\mathbb{R}^n$, the {\em restriction} of $\mathcal{A}$ to $U$ is the hyperplane arrangement in $U$ defined by
\[
\mathcal{A}|_{U}:=\{H \cap U\mid  H \in \mathcal{A}\And U\not\subseteq H\}.
\]
When $U$ has dimension $k$, $\mathcal{A}|_{U}$ is referred to as a {\em$k$-restriction} of $\mathcal{A}$. A more general setting and background on hyperplane arrangements can be found in the literature \cite{Orlik-Terao,Stanley}.

Let $[n]:=\{1,2,\ldots,n\}$ and ${[n]\choose k}:=\big\{I\subseteq [n]\mid \#I=k\big\}$. We consider $\mathbb{R}^{[n]\choose k}$ as the $n\choose k$-dimensional vector space over $\mathbb{R}$ with the corresponding components of its every vector indexed by the $k$-subsets $I\in {[n]\choose k}$. A subspace $U\in \mathrm{Gr}(k, n)$ can be specified as the row space of a $k\times n$ real matrix $A$, known as a {\em matrix representative} of $U$. It is well known that a $k$-dimensional subspace $U$ corresponds to a unique Pl\"ucker coordinate $\Delta(U)$ up to a nonzero scalar multiple. Specifically, associated with a  matrix representative $A$ of $U$, $\Delta(U):=\big(\Delta_{I}(U)\big)_{I\in{\binom{[n]}{k}}}$ is a vector in $\mathbb{R}^{[n]\choose k}$ whose every component $\Delta_{I}(U)$ is the $k\times k $ minor of $A$ formed from columns indexed by the $k$-subset $I$. 

It is worth remarking that the Pl\"ucker coordinate allows every $k$-dimensional subspace in $\mathbb{R}^n$ to define a hyperplane in $\mathbb{R}^{[n]\choose k}$. Recently, Liang, Wang and Zhao \cite{Liang-wang-zhao} observed this phenomenon and introduced the $k$-adjoint of $\mathcal{A}$ through its $k$-flats, extending the usual adjoint arrangement proposed by Bixby and Coullard \cite{Bixby-Coullard1988}.
\begin{defi}\label{k-adjoint-def}
{\rm 
Let $\mathcal{A}$ be a hyperplane arrangement in $\mathbb{R}^n$, $X\in L_{k}(\mathcal{A})$ and $I\in {[n]\choose k}$ for some $k=0,1,\ldots,n$. Setting
\[a_{I}(X)=(-1)^{\frac{1}{2}k(k+1) + \sum_{i\in I} i}\Delta_{[n]-I}(X),\]
we define the \emph{adjoint} $H(X)$ of $X$ as a hyperplane in $\mathbb{R}^{{[n]}\choose{k}}$ by
\[
H(X)=\left\{({x}_I)_{I\in{\binom{[n]}{k}}}\middle|\sum_{I\in{\binom{[n]}{k}}}a_I(X)\cdot {x}_I=0\right\}.
\]
Furthermore, the hyperplane arrangement $\mathcal{A}^{(k)}$ in $\mathbb{R}^{{[n]}\choose{k}}$ consists of the adjoints of all $k$-flats of $\mathcal{A}$, i.e.,
\[
\mathcal{A}^{(k)}=\big\{H(X) \mid X\in L_{k}(\mathcal{A})\big\},
\]
which is referred to as the \emph{$k$-adjoint} of $\mathcal{A}$. In particular, $\mathcal{A}^{(n)}$ is the empty arrangement.
}
\end{defi}
It is worth pointing out that when $\mathcal{A}$ is essential, there have been several important works. In particular, taking $k=n-1$ in Definition \ref{k-adjoint-def}, $\mathcal{A}^{(n-1)}$ is the usual adjoint arrangement. In this context, Fu and Wang \cite[Theorem 3.2]{FW2023} provided a unified classification of one-element extensions of $\mathcal{A}$ and its restrictions connecting all hyperplanes using the intersection lattice of $\mathcal{A}^{(n-1)}$. We extend this result to general subspaces and present a classification of all restrictions of a general hyperplane arrangement $\mathcal{A}$ regarding all $k$-dimensional subspaces.  The classification is achieved through the $\mathcal{A}$-matroid decomposition of the Grassmannian in Theorem \ref{Classification}, and through the $\mathcal{A}$-adjoint decomposition of the Grassmannian in Corollary \ref{Classification1}. When $\mathcal{A}$ is essential, this classification originated from \cite[Theorem 1.2]{Liang-wang-zhao}, and can also be described by the $k$-adjoint arrangement of $\mathcal{A}$. Recently, Liang, Wang and Zhao \cite{Liang-wang-zhao} introduced the adjoint decomposition of the Grassmannian $\mathrm{Gr}(k, n)$ via the $k$-adjoint arrangement, and then extended the classical Schubert and matroid decompositions of the Grassmannian $\mathrm{Gr}(k, n)$ to refined $\mathcal{A}$-Schubert and $\mathcal{A}$-matroid decompositions. Furthermore, they demonstrated that these three different methods lead to the same decomposition of the Grassmannian.
\begin{theorem}[\cite{Liang-wang-zhao}, Theorem 1.2]\label{LZW2025}
Let $\mathcal{A}$ be an essential hyperplane arrangement in $\mathbb{R}^n$. Then, the $\mathcal{A}$-matroid decomposition, the $\mathcal{A}$-adjoint decomposition and the refined $\mathcal{A}$-Schubert decomposition are exactly the same decomposition of the Grassmannian.
\end{theorem}

In this paper,  we primarily generalize the above result to general hyperplane arrangements. Detailed definitions of the three decompositions---namely, the $\mathcal{A}$-matroid decomposition, the $\mathcal{A}$-adjoint decomposition and the refined $\mathcal{A}$-Schubert decomposition---are presented in Section \ref{Sec-2}.
\begin{theorem}\label{main}
Let $\mathcal{A}$ be a hyperplane arrangement $\mathcal{A}$ in $\mathbb{R}^n$. Then, the $\mathcal{A}$-matroid decomposition, the $\mathcal{A}$-adjoint decomposition and the refined $\mathcal{A}$-Schubert decomposition are exactly the same decomposition of the Grassmannian.
\end{theorem}
\section{Three $\mathcal{A}$-decompositions of the Grassmannian}\label{Sec-2}
We work with a fixed hyperplane arrangement 
\[
\mathcal{A} = \{H_i: \langle\alpha_i, v\rangle=0\mid i=1,2, \ldots, m\}
\]
in $\mathbb{R}^n$, where the notation $\langle \cdot, \cdot\rangle$ represents the standard inner product in $\mathbb{R}^n$. In this section, we focus on reviewing the definitions of the three decompositions: the $\mathcal{A}$-matroid, the $\mathcal{A}$-adjoint and the refined $\mathcal{A}$-Schubert decompositions, and also provide a classification for all $k$-restrictions of $\mathcal{A}$ by using the $\mathcal{A}$-matroid decomposition in Theorem \ref{Classification}. Combining Theorem \ref{main}, this classification can also be equivalently described by the $\mathcal{A}$-adjoint decomposition. In particular, when $\mathcal{A}$ is essential, this classification can be more straightforward and explicitly characterized by the flats of the $k$-adjoint arrangement $\mathcal{A}^{(k)}$, as shown in Corollary \ref{Classification1}.

Matroid terminology follows Oxley's book \cite{Oxley}. For each subspace $U$ of $\mathbb{R}^n$, let $\beta_i = \mathrm{Proj}_U\alpha_i$ denote the orthogonal projection of $\alpha_i$ in $U$ for $i = 1, \ldots, m$. This enables us to construct a matroid $\mathfrak{M}_{\mathcal{A}}(U)$ on the ground set $[m]$ by defining its rank function on subsets $I \subseteq [m]$ as 
\[
\rk_{\mathfrak{M}_{\mathcal{A}}(U)}(I) := \dim\big(\mathrm{span}\{\beta_i \mid i \in I\}\big).
\]
Associated with the hyperplane arrangement $\mathcal{A}$, we further construct a subclass of matroids on $[m]$ by
\[
\mathbf{M}(\mathcal{A}):=\big\{\mathfrak{M}_{\mathcal{A}}(U) \mid U \in \mathrm{Gr}(k, n)\big\}.
\]  
We are now ready to introduce the $\mathcal{A}$-matroid decomposition of the Grassmannian, which was initially presented in \cite[Definition 3.2]{Liang-wang-zhao} for essential hyperplane arrangement $\mathcal{A}$.
\begin{defi}\label{A-matroid}
{\rm For each matroid $\mathfrak{M} \in \mathbf{M}(\mathcal{A})$, let
	\[
	\Omega_{\mathcal{A}}(\mathfrak{M}) := \left\{ U \in \mathrm{Gr}(k, n) \,\big|\, \mathfrak{M}_{\mathcal{A}}(U) = \mathfrak{M} \right\},
	\]
which is known as an {\em $\mathcal{A}$-matroid stratum}. Moreover, the natural decomposition
\begin{equation*}
\mathrm{Gr}(k, n) = \bigsqcup_{\mathfrak{M} \in \mathbf{M}(\mathcal{A})} \Omega_{\mathcal{A}}(\mathfrak{M})
\end{equation*}
is said to be	the \emph{$\mathcal{A}$-matroid decomposition} of the Grassmannian $\mathrm{Gr}(k, n) $.
}
\end{defi}
Let $\mathfrak{M}\in\mathbf{M}(\mathcal{A})$. Following Definition \ref{A-matroid}, if $U_1,U_2 \in \Omega_{\mathcal{A}}(\mathfrak{M})$, then we have 
\[
\mathfrak{M}_{\mathcal{A}}(U_1)=\mathfrak{M}_{\mathcal{A}}(U_2) = \mathfrak{M}.
\]
Additionally, every such matroid $\mathfrak{M}_{\mathcal{A}}(U)$ is geometrically realized by exactly the hyperplane arrangement $\mathcal{A}|_U$. It follows from  \cite[Proposition 3.6 in p. 425 ]{Stanley} that the intersection lattices $L(\mathcal{A}|_{U_1})$ and $L(\mathcal{A}|_{U_2})$ are isomorphic. Therefore, the $\mathcal{A}$-matroid decomposition in Definition \ref{A-matroid} effectively provides a classification of $L(\mathcal{A}|_{U})$ for all $k$-dimensional subspaces $U$ in $\mathbb{R}^n$. We summarize this finding in the following theorem, established by Liang, Wang and Zhao for an essential hyperplane arrangement $\mathcal{A}$ in their work following Theorem 1.2 of \cite{Liang-wang-zhao}.
\begin{theorem}\label{Classification}
Let $\mathcal{A}$ be a hyperplane arrangement, $U_1,U_2$ be two $k$-dimensional subspaces in $\mathbb{R}^n$, and $\mathfrak{M} \in \mathbf{M}(\mathcal{A})$. If  $U_1,U_2\in \Omega_{\mathcal{A}}(\mathfrak{M})$, then 
\[
L(\mathcal{A}|_{U_1})\cong L(\mathcal{A}|_{U_2}).
\]
\end{theorem}

We now proceed to introduce the $\mathcal{A}$-adjoint decomposition of the Grassmannian associated with $\mathcal{A}$. For each $X \in L(\mathcal{A})$, the {\em relative interior} of $X$ is defined as  
\begin{equation*}
X^{\circ} = X-\bigcup_{Y\subsetneq X\text{ in } L(\mathcal{A})} Y.
\end{equation*}  
\begin{defi}\label{A-adjoint}
{\rm For each $i \in \{0,1, \ldots, k\}$ and $P \in L\big(\mathcal{A}^{(k-i)}\big) $, let
\[
\mathcal{S}_{i, P}: = \left\{ U \in \mathrm{Gr}(k, n) \mid \dim(U \cap T) = i,  \Delta\big(U \cap (U^\perp+T^\perp)\big) \in P^\circ \right\},
\]
which is called an {\em$\mathcal{A}$-adjoint stratum}. Moreover, the decomposition
\begin{equation*}
\mathrm{Gr}(k, n) = \bigsqcup_{i \in \{0,1, \ldots, k\}} \ \bigsqcup_{P \in L(\mathcal{A}^{(k-i)})} \ \mathcal{S}_{i, P}.
\end{equation*}
is said to be the {\em$\mathcal{A}$-adjoint decomposition} of the Grassmannian $\mathrm{Gr}(k, n)$.
}
\end{defi}
Immediately, following Theorem \ref{main} and Theorem \ref{Classification}, we can easily observe that if $U_1,U_2\in \mathcal{S}_{i, P}$ for some flat $P$ of the $(k-i)$-adjoint arrangement $\mathcal{A}^{(k-i)}$ with $i\le k$, then 
\[
L(\mathcal{A}|_{U_1})\cong L(\mathcal{A}|_{U_2}).
\] 
Therefore, we can classify all $k$-restrictions of a general hyperplane arrangement $\mathcal{A}$ by utilizing all $i$-adjoint arrangements $\mathcal{A}^{(i)}$ with $i\le k$.

When $\mathcal{A}$ is essential, we have $T=\{o\}$. It follows that $\dim(U \cap T)=0$ for all $U\in\mathrm{Gr}(k, n)$ in this case. Then, the $\mathcal{A}$-adjoint decomposition of the Grassmannian 	$\mathrm{Gr}(k, n)$ is simplified to the form 
\[
\mathrm{Gr}(k, n)=\bigsqcup_{\substack{P\in L(\mathcal{A}^{(k)})}}\mathcal{S}_{0,P},
\]
which was first proposed by Liang, Wang and Zhao in \cite[Definition 3.1]{Liang-wang-zhao}. In this context, it is clear that for any two $k$-dimensional subspaces $U_1,U_2$ in $\mathbb{R}^n$, $U_1,U_2\in \mathcal{S}_{0,P}$ for some flat $P$ of the $k$-adjoint arrangement $\mathcal{A}^{(k)}$ if and only if  $\Delta(U_1),\Delta(U_2)\in P^{\circ}$. Following this, the above classification can be reduced to that $L(\mathcal{A}|_{U_1})$ and $L(\mathcal{A}|_{U_2})$ are isomorphic whenever $\Delta(U_1),\Delta(U_2)\in P^{\circ}$ with $P\in L(\mathcal{A}^{(k)})$. Furthermore, a basic decomposition for the whole space $\mathbb{R}^{\binom{[n]}{k}}$, associated with the flats of the $k$-adjoint arrangement $\mathcal{A}^{(k)}$, is given by
\[
\mathbb{R}^{\binom{[n]}{k}} = \bigsqcup_{P \in L(\mathcal{A}^{(k)})} P^{\circ}.
\]
This enables us to classify all $k$-restrictions of the essential arrangement $\mathcal{A}$ using the combinatorial structure $\mathcal{A}^{(k)}$, which was implicitly established in the work \cite[Theorem 1.2]{Liang-wang-zhao}. We conclude  the previous arguments with the following result.
\begin{coro}\label{Classification1}
Let $\mathcal{A}$ be a hyperplane arrangement, and $U_1,U_2$ be two $k$-dimensional subspaces in $\mathbb{R}^n$. If 
$U_1,U_2\in \mathcal{S}_{i, P}$ for some flat $P$ of the $(k-i)$-adjoint arrangement $\mathcal{A}^{(k-i)}$ with $i\le k$, then
\[
L(\mathcal{A}|_{U_1})\cong L(\mathcal{A}|_{U_2}).
\]
Particularly, when $\mathcal{A}$ is essential, if 
$\Delta(U_1),\Delta(U_2)\in P^{\circ}$ for some flat $P$ of the $k$-adjoint arrangement $\mathcal{A}^{(k)}$, then
\[
L(\mathcal{A}|_{U_1})\cong L(\mathcal{A}|_{U_2}).
\]
\end{coro}

The last decomposition generalizes the classical Schubert decomposition via the hyperplane arrangement $\mathcal{A}$.
Basic notations and concepts related to Schubert decompositions can be found in \cite{Fulton-1997}.

From now on, we assume $r(\mathcal{A})=r$. Let 
\[
\boldsymbol{F}\colon T = F_0 \subsetneq F_1 \subsetneq \cdots \subsetneq F_r = \mathbb{R}^n
\]
be a maximal chain of $L(\mathcal{A})$, where each \(F_j\) has dimension $n-r+j$. For each $i \in \{0,1, \ldots, k\}$ and a $(k-i)$-subset $\sigma=\{i_1,i_2,\ldots,i_{k-i}\}\subseteq [r]$ with elements ordered increasingly, 
we define a subset of the Grassmannian $\mathrm{Gr}(k, n)$ associated with $i$ and \(\sigma\) as   
\begin{align*}
	\Omega_{\boldsymbol{F}}(i, \sigma) 
	:= \Big\{ U \in \mathrm{Gr}(k, n) \;\Big|\;
	\begin{array}{l}
		\dim(U \cap T) = i, \\
		\dim(U \cap F_{i_l}) > \dim(U \cap F_{i_{l}-1}), 1 \leq l \leq k - i
	\end{array}
	\Big\}.
\end{align*}
Then, there is an immediate disjoint union
\begin{equation*}
\mathrm{Gr}(k, n)=\bigsqcup_{i \in \{0,1, \ldots, k\}}\bigsqcup_{\sigma\in\binom{[r]}{k-i}}\Omega_{\boldsymbol{F}}(i, \sigma).
\end{equation*}
Let $\mathbf{mc}\big(L(\mathcal{A})\big)$ denote the set of maximal chains of $L(\mathcal{A})$. 
It is clear that  
\begin{align*}
\mathrm{Gr}(k, n) 
&=
\bigsqcup_{i \in \{0,1, \ldots, k\}}\bigcap_{\boldsymbol{F} \in \mathbf{mc}(L(\mathcal{A}))}
\bigsqcup_{\sigma(\boldsymbol{F}) \in \binom{[r]}{k-i}}
\Omega_{\boldsymbol{F}}\big(i, \sigma(\boldsymbol{F})\big)\\[6pt]
& = \bigsqcup_{i \in \{0,1,\ldots, k\}}\bigsqcup_{\sigma_i\colon \mathbf{mc}(L(\mathcal{A})) \to \binom{[r]}{k-i}}
\bigcap_{\boldsymbol{F} \in \mathbf{mc}(L(\mathcal{A}))}
\Omega_{\boldsymbol{F}}\big(i, \sigma_i(\boldsymbol{F})\big),
\end{align*}
where $\sigma_i$ is an $i$-Schubert symbol of $\mathcal{A}$ presented in the upcoming definition.
\begin{defi}\label{A-Sucbert}
	Given $i \in \{0, \ldots, k\}$, an {\em$i$-Schubert symbol} of \( \mathcal{A} \) is a map $$ \sigma_i\colon\mathbf{mc}\big(L(\mathcal{A})\big)\longrightarrow \binom{[r]}{k-i}$$ such that  
\[
\Omega_{\mathcal{A}}(i, \sigma_i):=\bigcap_{\boldsymbol{F}\in\mathbf{mc}(L(\mathcal{A}))}\Omega_{\boldsymbol{F}}\big(i, \sigma_i(\boldsymbol{F})\big)\neq \varnothing.
\]
The {\em refined $\mathcal{A}$-Schubert decomposition} of the Grassmannian $\mathrm{Gr}(k, n)$ is given by
\begin{equation*}
\mathrm{Gr}(k, n)=\bigsqcup_{i \in \{0,1, \ldots, k\}}\bigsqcup_{\sigma_i\text{ is $i$-Schubert symbol}}\Omega_{\mathcal{A}}(i, \sigma_i).
\end{equation*}  
\end{defi}
Likewise, when $\mathcal{A}$ is essential, the refined $\mathcal{A}$-Schubert decomposition of the Grassmannian $\mathrm{Gr}(k, n)$ becomes the simplified version
\[
\mathrm{Gr}(k, n)=\bigsqcup_{\sigma_0\text{ is $0$-Schubert symbol}}\Omega_{\mathcal{A}}(0, \sigma_0),
\]
which agrees with Definition 3.3 of \cite{Liang-wang-zhao}. It should be noted that when $\mathcal{A}$ is the Boolean arrangement, the refined $\mathcal{A}$-Schubert decomposition corresponds to the common refinement of $n!$ permuted Schubert decompositions from \cite[Section 3]{Gelfand-Goresky-MacPherson-Serganova1987}.
\section{Proof of Theorem \ref{main}}
This section is devoted to proving Theorem \ref{main}. We still have some work to do before we achieve this. Let us first examine a close connection between a $k$-dimensional subspace and the elements in $L(\mathcal{A})$, which is the basis for defining the $k$-adjoint arrangement.
\begin{lemma}\label{lem-2.1}
Let $U\in\mathrm{Gr}(k, n)$. Then, for any $X \in L_{k-l}(\mathcal{A})$ with $l= \dim(U \cap T)$, we have 
\[
X \oplus\big(U \cap (U^\perp + T^\perp)\big) = \mathbb{R}^n\quad\text{ if and only if }\quad\Delta\big(U \cap (U^\perp + T^\perp)\big) \notin H(X).
\] 
\end{lemma}
\begin{proof}
Let $A$ and $B$ be matrix representatives of $U \cap (U^\perp+T^\perp)$ and $X$, respectively. Note that the dimensions of $U \cap (U^\perp+T^\perp)$ and $X$ are $k-l$ and $n-k+l$, respectively. This implies that the matrix $A$ together with the matrix $B$ yields an $n \times n$ matrix
\[
M = \left[\begin{array}{c} A \\ B \end{array}\right].
\]
Therefore, $\big(U \cap (U^\perp+T^\perp)\big)\oplus X= \mathbb{R}^n$ is equivalent to that the determinant $\det (M)$ is non-zero. By employing Laplace's expansion theorem on the first $k-l$ rows, we obtain
\[
\det (M) = \sum_{I\in\binom{[n]}{k-l}}(-1)^{\frac{1}{2}(k-l)(k-l + 1)+\sum_{i\in I}i}\Delta_I\big(U \cap (U^\perp+T^\perp)\big)\Delta_{[n]- I}(X).
\]
It immediately follows from the definition of the adjoint $H(X)$ in Definition \ref{k-adjoint-def} that $\big(U \cap (U^\perp+T^\perp)\big)\oplus X = \mathbb{R}^n$ if and only if $\Delta\big(U \cap (U^\perp+T^\perp)\big) \notin H(X)$.
\end{proof}

Let $U \in \mathrm{Gr}(k, n)$ with $\dim(U \cap T) = l$. It follows from Lemma \ref{lem-2.1} that for any $X \in L_{k-l}(\mathcal{A})$, $X \oplus \big(U \cap (U^\perp+T^\perp)\big) = \mathbb{R}^n$ if and only if $\Delta\big(U \cap (U^\perp+T^\perp)\big)\notin H(X)$. Based on this,  we define
\begin{align}\label{A}
L_U(\mathcal{A}):&=\big\{X\in L_{k-l}(\mathcal{A})\mid\Delta\big(U \cap (U^\perp+T^\perp)\big)\notin H(X)\big\}\notag\\[6pt]
&=\big\{X\in L_{k-l}(\mathcal{A})\mid X\oplus \big(U \cap (U^\perp+T^\perp)\big)=\mathbb{R}^n\big\}
\end{align}
and 
\begin{equation}\label{B}
L^U(\mathcal{A}):=\big\{X\in L_{k-l}(\mathcal{A})\mid \Delta\big(U \cap (U^\perp+T^\perp)\big)\in H(X)\big\}.
\end{equation}

Furthermore, to establish the relationship between the $\mathcal{A}$-matroid decomposition and the refined $\mathcal{A}$-Schubert decomposition, the next lemma is also required.
\begin{lemma}\label{lem-2}
Let $U$ be a subspace in $\mathbb{R}^n$. Then, for each subset $I$ of $[m]$, we have
\[
\dim\Big(U\bigcap\big(\bigcap_{i\in I} H_i\big)\Big)=\dim U-\dim\big(\mathrm{span}\{\beta_i\mid i\in I\}\big)
\]
and the rank function of the matroid $\mathfrak{M}_\mathcal{A}(U)$ satisfies
\[
\rk_{\mathfrak{M}_\mathcal{A}(U)}(I)=\dim U-\dim\Big(U\bigcap\big(\bigcap_{i\in I} H_i\big)\Big).
\]
\end{lemma}

\begin{proof}
When $|I| = 0$ (that is, $I$ is empty), we have that $\bigcap_{i\in I}H_i=\mathbb{R}^n$ and the result holds obviously. Consider the case $|I|\geq 1$. Notice that
\[
U\bigcap\big(\bigcap_{i\in I} H_i\big)=\big\{v\in U \mid \langle v,\beta_i\rangle = 0\text{ for }i\in I\big\}.
\]
The dimension formula follows immediately, as does the rank formula. 
\end{proof}
The following lemma serves as a key bridge between $\mathcal{A}$-matroid strata and the other two strata, and hence plays a central role in establishing connections among the $\mathcal{A}$-matroid, the $\mathcal{A}$-adjoint and the refined $\mathcal{A}$-Schubert decompositions of the Grassmannian.
\begin{lemma}\label{coro-2}
Let $U\in\mathrm{Gr}(k, n)$ with $\dim(U \cap T) = l$ and $I\subseteq [m]$. Then, the following statements are equivalent:
\begin{itemize}
\item[{\rm (1)}] $I$ is a basis of the matroid $\mathfrak{M}_\mathcal{A}(U)$.
\item[{\rm (2)}] $|I| = k-l$ and $\{\beta_i \mid i\in I\}$ is linearly independent.
\item[{\rm (3)}] $|I| = k-l$ and $\bigcap_{i\in I}H_i\in L_U(\mathcal{A})$.
\end{itemize}
\end{lemma}
\begin{proof}
Note that $\dim U=k$ and $\bigcap_{H \in \mathcal{A}} H=T$. We first assert that 
\[
U \cap (U^\perp+T^\perp)=\mathrm{span}\big\{\beta_i \mid i\in [m]\big\}.
\]
Clearly, $\mathrm{span}\big\{\beta_i \mid i\in [m]\big\}\subseteq U \cap (U^\perp+T^\perp)$. Furthermore, It follows from Theorem \ref{lem-2} that 
\[
\dim\big(\mathrm{span}\{\beta_i \mid i\in [m]\}\big) = k-l,
\]
which equals the dimension of $U \cap (U^\perp+T^\perp)$. Hence, the assertion holds. 

It is trivial that (1) is equivalent to (2). Now, we verify that (2) is equivalent to (3). Suppose $|I| = k-l$ and $\{\beta_i \mid i \in I\}$ is linearly independent. Then $\{\beta_i \mid i \in I\}$ forms a basis of $U \cap (U^\perp+T^\perp)$, $\{\alpha_i \mid i \in I\}$ is linearly independent and $\dim\big(\bigcap_{i\in I}H_i\big)=n - k + l$. For each $v\in\big(U \cap (U^\perp+T^\perp)\big)\bigcap\big(\bigcap_{i\in I} H_i\big)$, let $v = \sum_{i\in I}a_i\beta_i$. We derive
\[
\langle v,\alpha_j\rangle=\sum_{i\in I}a_i\langle\beta_i,\beta_j\rangle = 0,\quad \ j\in I.
\]
Since $\{\beta_i \mid i \in I\}$ is linearly independent, we deduce that $\det[\langle\beta_i,\beta_j\rangle]\neq 0$. It follows that all $a_i$ with $i \in I$ are zero. So, we have $v = 0$. As
\[
\dim\big(U \cap (U^\perp+T^\perp)\big) + \dim\big(\bigcap_{i\in I}H_i\big) = n,
\]
we conclude that $\big(U \cap (U^\perp+T^\perp)\big)\oplus\bigcap_{i\in I}H_i=\mathbb{R}^n$, and hence $\bigcap_{i\in I}H_i\in L_U(\mathcal{A})$.
	
Conversely, let $|I| = k-l$ and $\mathbb{R}^n =\big(U \cap (U^\perp+T^\perp)\big)\,\oplus\,\bigcap_{i\in I}H_i$. It is trivial that $\alpha_i$ with $i\in I$ are linearly independent. We prove the remainder by negation. Suppose $\{\beta_i \mid i \in I\}$ is linearly dependent. Then, there exists a non-zero vector $v\in U \cap (U^\perp+T^\perp)$ such that $\langle v,\beta_i\rangle = 0$ for all $i\in I$. For each $i \in I$, let $\alpha_i=\beta_i+\gamma_i$ with $\gamma_i\in U^{\perp}$. Then, we have 
\[
\langle v,\alpha_i\rangle=\langle v, \beta_i\rangle+\langle v,\gamma_i\rangle = 0,\quad\forall\,i\in I.
\]
Thus $v\in\big (U \cap (U^\perp+T^\perp)\big)\bigcap\big(\bigcap_{i\in I}H_i\big)$, contradicting the given direct sum.
\end{proof}

Based on the previous preparation, we now return to the proof of Theorem \ref{main}.
\begin{proof}[Proof of Theorem \ref{main}]
First, we show that the $\mathcal{A}$-matroid decomposition in Definition \ref{A-matroid} is the same as the $\mathcal{A}$-adjoint decomposition in Definition \ref{A-adjoint}. Let $\mathcal{S}_{i, P}$ be a $\mathcal{A}$-adjoint stratum presented in Definition \ref{A-adjoint}, where $i \in \{0,1, \ldots, k\}$ and $P\in L(\mathcal{A}^{(k-i)})$. Assume $U\in \mathcal{S}_{i, P}$. It implies that 
\[
\dim(U \cap T) = i\quad\And\quad\Delta\big(U \cap (U^\perp+T^\perp)\big)\in P^{\circ}.
\]
This is equivalent to that
\[
\Delta\big(U \cap (U^\perp+T^\perp)\big)\in P\quad\text{and} \quad \Delta\big(U \cap (U^\perp+T^\perp)\big)\notin Q
\]
for all $Q\in L(\mathcal{A}^{(k-i)})$ such that $Q\subsetneq P$. We further obtain that
\begin{align*}
P-\bigcup_{X\in L_U(\mathcal{A})} H(X)&= P^{\circ} =P-\bigcup_{X\in L(\mathcal{A}),\,P\nsubseteq H(X)} H(X)
\end{align*}
from \eqref{A} and 
\begin{align*}
\bigcap_{X\in L^{U}(\mathcal{A})} H(X)&= P =\bigcap_{X\in L(\mathcal{A}),\,P\subseteq H(X)} H(X)
\end{align*}
from \eqref{B}. It follows that 
\begin{align*}
L^{U}(\mathcal{A}) = \big\{X \in L_{k-i}(\mathcal{A}) \mid P \subseteq H(X)\big\}
\end{align*}	
and
\begin{align*}
 L_U(\mathcal{A}) = \big\{X \in L_{k-i}(\mathcal{A}) \mid P \nsubseteq H(X)\big\}.
\end{align*}	
Thus, when $U_1,U_2\in\mathcal{S}_{i,P}$, we have
\[
L^{U_1}(\mathcal{A}) = L^{U_2}(\mathcal{A})\quad\And\quad L_{U_1}(\mathcal{A}) = L_{U_2}(\mathcal{A}).
\]
Together with Lemma \ref{coro-2}, we arrive at $\mathfrak{M}_\mathcal{A}(U_1)= \mathfrak{M}_\mathcal{A}(U_2)$. Suppose $\mathfrak{M}\in\mathbf{M}(\mathcal{A})$ denotes the common matroid $\mathfrak{M}_\mathcal{A}(U)$ for all $U\in\mathcal{S}_{i,P}$. We deduce  $\mathcal{S}_{i,P}\subseteq\Omega_{\mathcal{A}}(\mathfrak{M})$.
	
Conversely, let $\Omega_{\mathcal{A}}(\mathfrak{M})$ be a $\mathcal{A}$-matroid stratum presented in Definition \ref{A-matroid}, where $\mathfrak{M} \in \mathbf{M}(\mathcal{A})$ with its rank $r(\mathfrak{M}) = k-i$. It follows that for any $U\in\Omega_{\mathcal{A}}(\mathfrak{M})$,
\[
\dim(U \cap T) = i\quad\And\quad\mathfrak{M}=\mathfrak{M}_{\mathcal{A}}(U).
\]
Let $U_1,U_2\in \Omega_{\mathcal{A}}(\mathfrak{M})$. Applying Lemma \ref{coro-2} again, we derive that $L_{U_1}(\mathcal{A}) = L_{U_2}(\mathcal{A})$. Consequently, $L^{U_1}(\mathcal{A}) = L^{U_2}(\mathcal{A})$. Let 
\[
P = \bigcap_{X \in L^{U_1}(\mathcal{A})}H(X) =  \bigcap_{X \in L^{U_2}(\mathcal{A})}H(X) \in L(\mathcal{A}^{(k-i)}).
\]
Then $\Delta\big(U_1 \cap (U_1^\perp+T^\perp)\big)$ and $\Delta\big(U_2 \cap (U_2^\perp+T^\perp)\big)$  lie in $P^\circ$. Then all members of $\Omega_{\mathcal{A}}( \mathfrak{M})$ lie in the same $\mathcal{A}$-adjoint stratum $\mathcal{S}_{i,P}$. This implies $\Omega_{\mathcal{A}}( \mathfrak{M}) \subseteq \mathcal{S}_{i,P}$.
	
Next, we show that the $\mathcal{A}$-matroid decomposition coincides with the refined $\mathcal{A}$-Schubert decomposition in Definition \ref{A-Sucbert}. Let $\Omega_\mathcal{A}(i, \sigma_i)$ be a stratum defined in Definition \ref{A-Sucbert}, where $i \in \{0,1,\ldots,k\}$ and $\sigma_i$ is an $i$-Schubert symbol of $\mathcal{A}$. Let $U \in \Omega_\mathcal{A}(i, \sigma_i)$. 
For each flag $\boldsymbol{F} \in\mathbf{mc}\big(L(\mathcal{A})\big)$, write
\(\boldsymbol{F}\colon T = F_0 \subsetneq F_1 \subsetneq \cdots \subsetneq F_r = \mathbb{R}^n\)
and
$\sigma_i(\boldsymbol{F}) = \big\{i_1(\boldsymbol{F}), \ldots, i_{k-i}(\boldsymbol{F})\big\}$.
Then we have $\dim (U\cap T) = i$ and
\[
\dim \left(U\cap F_{i_l(\boldsymbol{F})}\right) > \dim \left(U\cap F_{i_l(\boldsymbol{F})-1}\right)
\]
for any $l\in \{1, 2,\ldots, k-i\}$. Therefore, for each $X \in L_{k-i}\big(\mathcal{A}\big)$, we further have that $\big(U \cap (U^\perp+T^\perp)\big)\oplus X=\mathbb{R}^n$ if and only if there exists some flag $\boldsymbol{F}\in\mathbf{mc}\big(L(\mathcal{A})\big)$ with $\sigma(\boldsymbol{F}) = \{r-k+i+1, \ldots, r\}$ such that $X = F_{r-k+i}$. By \eqref{A} again, we derive
\[
L_U(\mathcal{A}) =\big\{F_{r-k+i} \mid \boldsymbol{F}\in\mathbf{mc}\big(L(\mathcal{A})\big)\,\And\, \sigma(\boldsymbol{F}) = \{r-k+i+1, \ldots, r\}\big\}.
\]
Thus for $U_1,U_2\in \Omega_\mathcal{A}(i,\sigma_i)$, we conclude $L_{U_1}(\mathcal{A}) = L_{U_2}(\mathcal{A})$. 
Combining Lemma \ref{coro-2}, we deduce $\mathfrak{M}_{\mathcal{A}}(U_1) = \mathfrak{M}_{\mathcal{A}}(U_2)$. Suppose $\mathfrak{M}\in\mathbf{M}(\mathcal{A})$ denotes the common matroid $\mathfrak{M}_\mathcal{A}(U)$ for all $U\in\Omega_\mathcal{A}(i,\sigma_i)$. We arrive at $\Omega_\mathcal{A}(i,\sigma_i)\subseteq\Omega_{\mathcal{A}}(\mathfrak{M})$.

Conversely, let $\Omega_{\mathcal{A}}(\mathfrak{M})$ be a $\mathcal{A}$-matroid stratum for some $\mathfrak{M} \in \mathbf{M}(\mathcal{A})$ with rank $r(\mathfrak{M}) = k-i$, and $U \in \Omega_{\mathcal{A}}(\mathfrak{M})$. Then $\dim(U \cap T) = i$ by Lemma \ref{lem-2}. Furthermore, for each subset $I \subseteq [m]$, Lemma \ref{lem-2} implies that
\[
\dim\Big(U\bigcap\big(\bigcap_{i\in I} H_i\big)\Big) = k-\rk_\mathfrak{M}(I).
\]
Therefore, for each maximal chain \(\boldsymbol{F}\colon T = F_0 \subsetneq F_1 \subsetneq \cdots \subsetneq F_r = \mathbb{R}^n\), 
the dimension $\dim(U \cap F_i)$ is determined by the matroid $\mathfrak{M}$ for $i = 1, \ldots, r$. This implies that the $i$-Schubert symbol $\sigma_i$ of $\mathcal{A}$ such that $U \in \Omega_\mathcal{A}(i,\sigma_i)$ is determined by the matroid $\mathfrak{M}$.  Hence, we obtain that all members of $\Omega_{\mathcal{A}}(\mathfrak{M})$ are in the same stratum $\Omega_{\mathcal{A}}(i,\sigma_i)$, i.e., $\Omega_{\mathcal{A}}(\mathfrak{M}) \subseteq \Omega_{\mathcal{A}}(i,\sigma_i)$. This completes the proof.
\end{proof}
\section*{Acknowledgements}
The work is supported by National Natural Science Foundation of China (12301424), and by Guangdong Basic and Applied Basic Research Foundation (2025A1515010457).


\begin{thebibliography}{99}\setlength{\itemsep}{-.0mm}
\bibitem{Bixby-Coullard1988}
R. E. Bixby and C. R. Coullard. Adjoints of binary matroids. European J. Combin. 9 (1988), 139--147.

\bibitem{Chen-Fu-Wang-2021}
B. Chen, H. Fu and S. Wang. Parallel translates of represented matroids. Adv. in Appl. Math. 127 (2021), Paper No. 102176.

\bibitem{FW2023}
H. Fu, S. Wang. Modifications of hyperplane arrangements. J. Combin. Theory Ser. A  200 (2023), Paper No. 105797.

\bibitem{Fulton-1997}
W. Fulton. Young Tableaux. London Mathematical Society Student Texts, 35, Cambridge Univ. Press, Cambridge, 1997.

\bibitem{Gelfand-Goresky-MacPherson-Serganova1987}
I. M. Gel'fand, R. M. Goresky, R. D. MacPherson and V. V. Serganova. Combinatorial geometries, convex polyhedra, and Schubert cells. Adv. in Math. 63 (1987), 301--316.

\bibitem{Liang-wang-zhao}
W. Liang, S. Wang and C. Zhao. $k$-adjoint of hyperplane arrangements.  2025, arXiv:2412.06633v2.

\bibitem{Orlik-Terao}
P. Orlik and H. Terao. Arrangements of Hyperplanes. Grundlehren der mathematischen Wissenschaften, 300, Springer, Berlin, 1992.

\bibitem{Oxley}
J. Oxley. Matroid Theory. Second edition, Oxford University Press, New York, 2011.

\bibitem{Stanley}
R. P. Stanley. An introduction to hyperplane arrangements. In: E. Miller, V. Reiner, B. Sturmfels (Eds.), Geometric Combinatorics, IAS/Park City Math. Ser. vol. 13, Amer. Math. Soc. Providence, RI, 2007, pp. 389--496.

\bibitem{Whitney1932}
H. Whitney. A logical expansion in mathematics. Bull. Amer. Math. Soc. 38 (1932), 572--579.
\end{thebibliography}
\end{document}